\documentclass[a4paper,reqno,12pt]{amsart}

\usepackage{amssymb,amsthm,amsmath}
\usepackage{xcolor}
\usepackage{hyperref}
\hypersetup{
	colorlinks=true,
	linkcolor=blue,
	citecolor=red,
}
\usepackage[left=3cm,right=2.7cm,top=3cm,bottom=2.7cm,twoside]{geometry}

\newtheorem{theorem}{Theorem}[section]

\newtheorem{corollary}[theorem] {Corollary}
\newtheorem{defin}[theorem]{Definition}

\newtheorem{lemma} [theorem]{Lemma}

\newtheorem{proposition}[theorem]{Proposition}
\newtheorem{rem}[theorem]{Remark}

{\theoremstyle{definition}
	\newtheorem{example}[theorem]{Example}
}

\newtheorem{thm}{Theorem}

\numberwithin{equation}{section}

\makeatletter
\@namedef{subjclassname@2020}{%
	\textup{2020} Mathematics Subject Classification}
\makeatother

\begin{document}

\title[Tumura-Clunie Differential Equations with Applications]{Tumura-Clunie Differential Equations with Applications to Linear ODE's}

\author{M. A.~Zemirni and Z.~Latreuch}

\begin{abstract}
In this paper, we study nonlinear differential equations of Tumura-Clunie type, $ f^n + P(z, f) = h, $ where \( n \geq 2 \) is an integer, \( P(z, f) \) is a differential polynomial in \( f \) of degree \( \gamma_P \leq n - 1 \) with small functions as coefficients, and \( h \) is a meromorphic function. Assuming that $ h $ satisfies a linear differential equation of order $ p\le n $ with rational coefficients, we establish a result that classifies the meromorphic solutions \( f \) into two cases based on the distribution of their zeros and poles. This result is then applied to study the zeros and critical points of entire solutions to certain higher-order linear differential equations, thereby extending some known results in the literature.

\medskip
\noindent
\textbf{Keywords.} Linear differential equation, meromorphic functions, Nevanlinna theory, nonlinear differential equation, Tumura-Clunie equation

\medskip
\noindent
\textbf{MSC 2020.} Primary 34M05; Secondary 30D35.
\end{abstract}

\maketitle

\renewcommand{\thefootnote}{}
\footnotetext[1]{Corresponding author: Z.~Latreuch.}


\section{Introduction}

We assume that the reader is familiar with the fundamental results and standard notations of Nevanlinna theory \cite{L,H}, including the notions of order, hyper-order, and convergence exponent of zeros. For instance, a meromorphic function \(\varphi\) is said to be a \textit{small} function with respect to a meromorphic function \( f \) if the Nevanlinna characteristic \( T(r,\varphi) \) of \( \varphi \) satisfies \( T(r,\varphi) = S(r,f) \), where the notation \( S(r,f) \) denotes any quantity satisfying \( S(r,f) = o(T(r,f)) \) as \( r \to +\infty \), possibly outside an exceptional set of such \( r \) of finite linear measure.  

A \textit{differential polynomial} \( P(z, f) \) in  a meromorphic function \( f \) is defined~by
\begin{equation}\label{def:P}
	P(z, f) := \sum_{j=1}^\ell a_{j}\ f^{n_{j0}} (f')^{n_{j1}} \cdots (f^{(k)})^{n_{jk}},
\end{equation}
where the powers \(n_{j0}, n_{j1}, \ldots, n_{jk}\) are non-negative integers and the coefficients \( a_j \) are meromorphic functions.
The \textit{degree} of \( P(z, f) \) is given by
\begin{equation*}
	\gamma_P := \max_{1\le j\le \ell} \sum_{s=0}^{k} n_{js}.
\end{equation*}

In this paper, we consider transcendental meromorphic solutions \( f \) of nonlinear differential equations of \textit{Tumura–Clunie type},
\begin{equation}\label{nde}
	f^n + P(z, f) = h(z),
\end{equation}
where \( n \geq 2 \) is an integer, \( h \) is a meromorphic function, and \( P(z, f) \) is given in \eqref{def:P} with \( 1 \leq \gamma_P \leq n - 1 \) and the coefficients \( a_j \) are small with respect to \( f \).

Equations of the form \eqref{nde} have appeared frequently in the literature in the context of Tumura–Clunie type theorems \cite{MS,R,Y,Yi}, originating from the classical result of Tumura and Clunie \cite{C,T}, which is extended later by Hayman \cite{H}. In the early 2000s, a direct investigation of the solutions of \eqref{nde} was initiated \cite{HKL,Y2}, particularly for the case where $ P(z,f) $ is a linear differential polynomial, i.e., \( \gamma_P = 1 \). Regarding the general form of $ P(z,f) $, several authors \cite{Li2,LYZ,Liao,LLW,YL,Z} have considered the equation \eqref{nde}, but with \( h \) having the specific form
\begin{equation}\label{h}
	h(z) = p_{1}(z) e^{\alpha_{1}(z)} + p_{2}(z) e^{\alpha_{2}(z)},
\end{equation}
where $ p_1$ and $ p_2 $ are small functions with respect to \( f \), $ \alpha_1 $ and $ \alpha_2 $ are non-constant entire functions. Later on, the function $ h $ is assumed to belong to larger classes of meromorphic functions than the class consisting  of the function in \eqref{h}. Essentially, there have been two directions in extending this class.  One approach is to consider that $ h $ is a combination of several exponential terms: 
\begin{eqnarray}\label{ngeq3}
	h(z) = \sum_{j=1}^p p_j(z) e^{\alpha_j(z)},
\end{eqnarray}
where $ p\ge 3 $, \( p_j \) are rational functions and \( \alpha_j \) are non-constant polynomials \cite{LM, CL, FC}. Partial classifications for solutions of \eqref{nde} have been obtained under the assumption \( N(r, f) = S(r, f) \). 
The second approach is to consider that \( h \) is a solution of a second-order linear differential equation with rational coefficients  \cite{HLWZ}. In this case, several asymptotic estimations for the growth of solutions are obtained, and  earlier results related to the form in \eqref{h} are refined.  Theorem~1.3 in \cite{HLWZ}, which is of  main interest in this note, reads as follows: \emph{If \( f \) is a transcendental meromorphic solution to \eqref{nde} and \( h \) satisfies a second-order linear differential equation with rational coefficients, then either \( f \) has only finitely many zeros and poles, or}
\[
T(r, f) \leq \frac{2}{n-\gamma_P} \left( \overline{N}(r,\infty, f)+\overline{N}\left(r, 0,{f}\right)\right)  + N(r,\infty, f)   + S(r, f).
\]

In this note, we explore an additional approach to possibly extend  further the previously mentioned results concerning equations \eqref{nde}. In fact, we consider that $ h $  is a meromorphic solution to a higher-order  linear differential equation,
\begin{equation}\label{odeh}
	h^{(p)} + r_{p-1}(z) h^{(p-1)} + \cdots + r_1(z) h' + r_0(z) h = r_p(z),
\end{equation}
where  \( p \geq 1 \), and \( r_0 \not\equiv 0 , r_1, \dots, r_{p-1}, r_p \) are rational functions. Under this consideration for $ h $, with the additional assumption $ p\le n $, we establish an extension of \cite[Theorem~1.3]{HLWZ}.  Consequently, this extension is used to study the zero and critical points of solutions of linear differential equations
\begin{eqnarray} \label{L}  
	f^{(n)} + p_{n-1}(z)f^{(n-1)} + \ldots + p_1(z)f' + h(z)f = 0,  
\end{eqnarray}  
where \( p_1, \ldots,p_{n-1} \) are polynomials, and \( h \) is a transcendental entire solution of \eqref{odeh} with polynomial coefficients and $ p\le n $.

The main results concerning \eqref{nde} and \eqref{L} are presented in Section~\ref{sec:results}, while their proofs are given in Section~\ref{sec:proofs}.

\section{Results}\label{sec:results}
\subsection{Tumura-Clunie equation.}
We will prove the following theorem, which is in the style of Tumura-Clunie type theorems and extends \cite[Theorem~1.3]{HLWZ}. The integer \( p\) in \eqref{odeh} is taken to be  minimal in the sense that \( h \) does not satisfy any linear differential equation of order \(\leq p - 1\) with rational coefficients. For example, the function
\[
h(z) = e^z + e^{2z} + e^{3z}
\]
satisfies a linear differential equation of order \( p = 3 \) with rational coefficients, yet it does not satisfy any such equation of second order. 

\begin{theorem}\label{t1}
	Let $n\ge p$, $\gamma_P \le n-1$ and let $f$ be a transcendental meromorphic solution of \eqref{nde}, where $h$ is a meromorphic solution of \eqref{odeh}. Then one of the following holds:
	\begin{enumerate}
		\item [(1)]  $f=qe^{\alpha}$, where $q$ is non-zero rational function and $\alpha$ is non-constant polynomial, and we have
		$$
		T(r,h) = nT(r,f) + S(r,f).
		$$
		In particular, if \( r_0, \dots, r_{p-1}\) are entire, then $q$ is constant.
		
		\item [(2)] $ f $ satisfies
		$$
		T(r,f) \le N(r,\infty,f)+\frac{p}{n-\gamma_P} \left(\overline{N}(r,\infty,f) + \overline{N}(r,0,f)\right) + S(r,f).
		$$
	\end{enumerate}
\end{theorem}

\begin{rem}\label{rem:finitecoeff}
	If the coefficients of $ P(z,f) $ are all rational functions, then $ S(r,f) $ in Theorem~\ref{t1}(1) is $ O(\log r) $ without exceptional set. In addition, if $ \rho(f)<+\infty $, then $ S(r,f) $ in Theorem~\ref{t1}(2) is $ O(\log r) $ without exceptional set.
\end{rem}

Notice that exponential combinations in \eqref{ngeq3} are indeed  meromorphic solutions to equations of the form \eqref{odeh}. In contrast to the functions in \eqref{ngeq3}, the meromorphic solutions of \eqref{odeh} may have non-integer order.  Therefore, the results obtained in \cite{LM, CL, FC} might be extended to all meromorphic solutions, rather than only those with a few poles. 

The proof of Theorem~\ref{t1} is different from that of Theorem~1.3 in \cite{HLWZ}. For instance, in the case $ p=2 $, our proof simplifies the one in \cite{HLWZ}. 

Theorem~\ref{t1} can also be stated in terms of deficiencies:
\begin{equation*}\label{def}
	\delta(a,f):= 1- \limsup_{r\to\infty}\frac{N(r,a,f)}{T(r,f)},\quad \Theta(a,f):= 1- \limsup_{r\to\infty}\frac{\overline{N}(r,a,f)}{T(r,f)},
\end{equation*}
where $a\in \widehat{\mathbb{C}}$, and we have,
\begin{corollary}\label{coro}
	Let $n\ge p$, $\gamma_P \le n-1$ and let $f$ be a transcendental meromorphic solution of \eqref{nde}, where $h$ is a meromorphic solution of \eqref{odeh}. If $ f $ is not of the form $ qe^\alpha $, where $ q $ is non-zero rational and $ \alpha $ is non-constant polynomial, then 
	\begin{equation}\label{cond}
		\Theta(0, f)+\Theta(\infty, f)+\frac{n-\gamma_P}{p} \delta(\infty, f) \le 2.
	\end{equation}
\end{corollary}
\begin{proof}
	Assume the contrary of \eqref{cond}. Then for any $\varepsilon\in (0,1)$ sufficiently small, there is an $R>0$ such that 
	$$
	\frac{p}{n-\gamma_P}\left( \overline{N}(r,0,f) +\overline{N} (r,\infty,f) \right)+ N(r,\infty, f)   \le (1-\varepsilon) T(r,f),\quad r>R.
	$$
	This together with the case (2) of Theorem~\ref{t1} gives $T(r,f)=S(r,f)$, which is absurd. Thus, the case (2) in Theorem~\ref{t1} cannot hold, which means that the only possible situation is the case (1) of Theorem~\ref{t1}, that is, $ f $ is of the form $ qe^\alpha $, which is a contradiction.
\end{proof}

For the particular interest when $h$ is rational or has non-integer order, Theorem~\ref{t1} shows also that: \textit{if $h$ is either a rational or transcendental with $ \rho(h) \not\in \mathbb{N} $, then \eqref{cond} holds.}

%


The following example  illustrates Theorem~\ref{t1} and Corollary~\ref{coro} when $h$ satisfies $3^{\text{rd}}$ order ODE.
\begin{example}\label{sharpness-of-cor} 
	The meromorphic function $f(z)=e^{2z}/(e^z-1)$ has no zeros and it satisfies
	\begin{eqnarray*}
		T(r,f) &=& 2r/\pi + O(1),\\
		\overline{N}(r,\infty,f) &=& N(r,\infty,f) = r /\pi + O(1).
	\end{eqnarray*}
	Thus $\Theta(\infty,f)=\delta(\infty,f)= 1/2$ and $\Theta(0,f)=1$. Moreover, $f$ solves the
	equation
	$$
	f^4+\frac{1}{6}f'''-3 f'' + \frac{107}{6}f'-35 f = e^{2 z}\left(e^{2 z}+4 e^{z}+10\right).
	$$
	Here, $h(z)= e^{2 z}\left(e^{2 z}+4 e^{z}+10\right)$ is a solution of the  equation $h'''-9h''+26h'-24h=0$.
	In this case, we have $\Theta(0,f)+\Theta(\infty,f)+ \frac{4-1}{3}\delta(\infty,f)= 2$. This shows the sharpness of \eqref{cond}.
\end{example}

\subsection{Applications to linear differential equations.}

This section is devoted to studying linear differential equations given in \eqref{L}.

It is known that every nontrivial solution of \eqref{L} is transcendental entire function, and it is shown in \cite[Theorem~2.3]{HYZ} that  every such solution satisfies 
\begin{equation*}
	\log T(r,f) \asymp \log M(r,h),\quad r\notin E, 
\end{equation*}
where $ E \subset [0,+\infty)$ is a set of finite linear measure.  In particular, we have $ \rho_2(f) = \rho(h) $.

It was shown in \cite[Theorem 6]{HILT} that if \( f \) is a nontrivial solution of  
\begin{equation}\label{Eq}
	f'' + h(z) f = 0,
\end{equation}
where \( h \) is a transcendental entire function satisfying \( \overline{N}(r, 0, h) \not= S(r, h) \) and  
\begin{equation*}
	K\, \overline{N}(r, 0, h) \leq T(r, h) + S(r, h), \quad K > 2,
\end{equation*}
then \( f \) satisfies
\begin{equation*}
	\limsup_{r \to \infty} \frac{\overline{N}\left(r, 0, f f'\right)}{\overline{N}\left(r, 0, h\right)} \geq \frac{K - 2}{2} > 0.
\end{equation*}
In particular, 
\[
\rho(h) \leq \max \left\{ \overline{\lambda}(f), \, \overline{\lambda}(f') \right\}.
\]
Motivated by this result, one may ask whether a similar conclusion can be drawn regarding the zeros and critical points of solutions to the equation \eqref{L}.  
In fact, by applying Theorem~\ref{t1}, we prove the following result, which gives another general perspective to \cite[Theorem 6]{HILT}.

\begin{theorem}\label{Osc}
	Let $ f $ be a nontrivial solution of \eqref{L}.
	Then one of the following holds:
	\begin{enumerate}
		\item[(1)] 	$ f = e^{G} $, where $ G' = e^{\alpha} $ and
			\( \alpha \) is a polynomial with $\deg\alpha=\rho(h)$.
			
			\item[(2)] $ f $ satisfies
			\begin{eqnarray}\label{res}
				T(r, h) \leq np \, \overline{N}\left( r,0,{ff'} \right)+ o \left(\overline{N}\left( r, 0,{f} \right)\right)  + o\left(T(r,h)\right),
			\end{eqnarray}
			outside an exceptional set of finite linear measure; and if  $\overline{\lambda}(f) < +\infty $, then 
			\begin{equation}\label{res1}
				T(r,h) \le np \overline{N}(r,0,ff') + O(\log r).
			\end{equation}
		\end{enumerate}
	\end{theorem}

	The idea of using Tumura-Clunie equation to study the solutions of the equation \eqref{Eq} is used in \cite{Z} for  $ h $ having the form $ b_1(z) e^{p_1(z)} + b_2(z) e^{p_2(z)} +b_3$, where $ p_1 $ and $ p_2 $ are non-constant polynomials having the same degree, and $ b_1,b_2 $ and $ b_3 $ are polynomials, and the solutions $ f $ having few zeros in the sense $ \lambda(f)<\deg(p_1) $.
	
	
	The following proposition reveals additional information when \( h \) satisfies a non-homogeneous linear differential equation \eqref{odeh}.

	\begin{proposition}\label{Pro}
		Let \( h \) be a transcendental solution of \eqref{odeh} with \( r_p \not\equiv 0 \). Then \( \overline{N}(r, 0, h) \neq S(r, h) \), and  Theorem~\ref{Osc}(1) does not occur.
	\end{proposition}

	\begin{proof}
		Assume, for contradiction, that \( \overline{N}(r, 0, h) = S(r, h) \).  Since \( r_p \not\equiv 0 \), we can write \eqref{odeh} as  
		\[
		\frac{1}{h} = \frac{1}{r_p(z)} \left( \frac{h^{(p)}}{h} + \sum_{j=0}^{p-1} r_j(z) \frac{h^{(j)}}{h} \right),
		\]  
		which implies \( m(r, 1/h) = O(\log r) \). Combined with the assumption \( \overline{N}(r, 0, h) = S(r, h) \), this leads to a contradiction, namely \( T(r, h) = S(r, h) \), which is impossible since \( h \) is transcendental.
		
		Next, suppose, again for contradiction, that \( f \) satisfies Case~(1) of Theorem~\ref{Osc}, that is, \( f = e^G \), where \( G' = e^\alpha \) and \( \alpha \) is a polynomial. Substituting \( f = e^G \) into equation \eqref{L} gives  
		\[
		h(z) = -e^{n\alpha(z)} + \sum_{j=1}^{n-1} \beta_j(z) e^{j\alpha(z)},
		\]  
		for some polynomials \( \beta_j \). Substituting this expression for \( h \) into \eqref{odeh} and applying Borel's theorem \cite[Theorem 1.51]{YY} leads to a contradiction, since \( r_p \not\equiv 0 \).
	\end{proof}
	

	The following two examples illustrate the both cases of Theorem~\ref{Osc}. 
	\begin{example}  
		The function \( f(z) = e^{e^z} \)  solves the differential equation  
		\[
		f'' - (e^z + e^{2z})f = 0,  
		\]  
		where \( h(z) = -(e^z + e^{2z}) \) satisfies the auxiliary equation  
		\[
		h'' - 3h' + 2h = 0.
		\]  
	\end{example}  
	
	\begin{example}[\cite{BLL}]\label{Ex2}
		For any \( \gamma \in \mathbb{C} \setminus \{0\} \), the differential equation  
		\[
		f^{\prime \prime} - \left(\gamma^2 e^z - \frac{\gamma}{2} e^{z / 2} + \frac{1}{4}\right) f = 0
		\]  
		admits two linearly independent solutions.
		\[
		f_1(z) = \exp\left(2 \gamma e^{z / 2} - z / 2\right), \quad f_2(z) = \left(4 \gamma e^{z / 2} + 1\right) \exp\left(-2 \gamma e^{z / 2} - z / 2\right).
		\]
		Here, the coefficient \( h(z) = -\left(\gamma^2 e^z - \frac{\gamma}{2} e^{z/2} + \frac{1}{4}\right) \) solves the ODE
		\[
		h'' - \frac{3}{4}h' + \frac{1}{4}h = \frac{1}{16},
		\]
		and satisfies 
		\[
		T(r, h) = \frac{r}{\pi} + o(r), \quad \text{as } r \to \infty.
		\]
		Additionally, we have
		\[
		\begin{aligned}  
			N\left(r, 0,f_1\right) &= 0, & N\left(r, 0,f_1'\right) &= \frac{r}{2\pi} + o(r), \quad r \to \infty, \\  
			N\left(r, 0,f_2\right) &= \frac{r}{2\pi} + o(r), & N\left(r, 0,f_2'\right) &= \frac{r}{\pi} + o(r), \quad r \to \infty.  
		\end{aligned}  
		\]  
		This clearly shows that \( f_1 \) and \( f_2 \) satisfy the inequality \eqref{res}.  
	\end{example}


	\begin{corollary}\label{cor}
		Under the hypotheses of Theorem~\ref{Osc}, and assuming \( r_p \not\equiv 0 \), we have: If \( \overline{\lambda}(f) < \rho(h) \), then 
		\begin{eqnarray}\label{res2}
			T(r, h) \leq np \overline{N}\left( r, 0,f' \right)  + S(r,h).
		\end{eqnarray} Moreover, \( f \) takes the form \( f = \pi e^g \), where \( \pi \) and $g$ are entire function such that \( \rho(\pi) < \rho(g) \), and
		\begin{eqnarray}\label{res3}
			T(r, g') \leq n\overline{N}\left( r, 0,g'\right) + S(r, g').
		\end{eqnarray}
	\end{corollary}
	
	Corollary~\ref{cor} provides more details on the form of solutions of \eqref{L} when  satisfies \( \overline{\lambda}(f) < \rho(h) \). 
	
	The following theorem is known  as the \(\frac{1}{16}\)-theorem.
	
	\begin{thm}[{\cite[Theorem 3.3]{BLL}}]\label{Comp0}
		Let \( h = e^{P} + Q \), where \( P \) is a non-constant polynomial, and \( Q \) is an entire function satisfying \( \rho(Q) < \deg P \). If \eqref{Eq} admits a nontrivial solution \( f \) with \( \lambda(f) < \deg P \), then \( f\) has no zeros and \( Q \) reduces to a polynomial of the form  
		\[
		Q(z) = -\frac{1}{16} P'(z)^2 + \frac{1}{4} P''(z).
		\]
		Moreover, \eqref{Eq} has two linearly independent zero-free solutions.
	\end{thm}
	
	The following consequence of Corollary~\ref{cor} provides additional information on the zero-free solutions in Theorem~\ref{Comp0}
	
	\begin{corollary}\label{Comp}
		The zero-free solutions \( f \) in Theorem~\ref{Comp0} are of the form \( f = e^{g} \), where \( g \) is an entire function satisfying \( \lambda(g') = \lambda(g) = \deg P \).
	\end{corollary}
	
	The following example was given in \cite{HILT} and illustrates Corollary~\ref{Comp}.
	\begin{example}
		The functions \( f_j(z) = e^{g_j(z)} \), where  
		\[
		g_j(z) = (-1)^j 2i \exp \left( \frac{z}{2} \right) - \frac{z}{4}, \quad j = 1, 2,
		\]
		are linearly independent solutions of  
		\[
		f'' + \left( e^z - \frac{1}{16} \right) f = 0.
		\]
		Clearly, \( \overline{\lambda}(g_j') = \overline{\lambda}(g_j) = 1 \).
	\end{example}

	\section{Proofs}\label{sec:proofs}
	\subsection{Proof of Theorem~\ref{t1}.}

	In order to prove Theorem~\ref{t1}, we need to define the following notations. For $a\in\widehat{\mathbb{C}}$ and $k\in\mathbb{N}$, we denote the  counting function of the $a$-points of a meromorphic $f$ of multiplicity at most $k$ by $n_{k)}(r,a,f)$, each $a$-point being counted according to its multiplicity, and $N_{k)}(r,a,f)$ being its corresponding integrated counting function. Further, $n_{(k}(r,a,f)$ denotes the counting function of the $a$-points of $f$
	of multiplicity at least $k$, again each $a$-point being counted according to its multiplicity and $N_{(k}(r,a,f)$ being its corresponding integrated counting function. Then $N(r,a,f)=N_{k)}(r,a,f)+N_{(k+1}(r,a,f)$. We also make use of the notation $n_{(k)}(r,a,f)$ for the counting function of the $a$-points of $f$
	of multiplicity exactly $k$.
	
	We need the following lemmas.
	\begin{lemma}\label{zerospoles}
		Let $g$ be a non-trivial meromorphic solution of 
		$$
		g^{(p)}+A_{p-1} g^{(p-1)}+\cdots + A_0 g=0,
		$$
		where $A_s$ ($s=0,\dots,p-1$) are meromorphic functions. Then
		\begin{eqnarray*}
			n(r,\infty, f) &\le & \sum_{s=0}^{p-1} n_{(p-s} (r,\infty,A_s), \\
			n_{(p} \left(r,0,f \right) &\le &\sum_{s=0}^{p-1} n_{(p-s} (r,\infty,A_s).
		\end{eqnarray*}
		In particular, if all the coefficients are entire functions, then $f$ is entire with zeros of multiplicity at most $p-1$.
	\end{lemma}
	\begin{proof}
		Write the equation in the form
		$$
		1+\sum_{s=0}^{p-1}A_{s} \frac{g^{(s)}}{g^{(p)}} =0. 
		$$
		If $z_0$ is a pole of any multiplicity or a zero of multiplicity at least $p$ of $g$, then $z_0$ is a zero of multiplicity $p-s$ of $\frac{g^{(s)}}{g^{(p)}}$, then $z_0$ must be a pole of multiplicity at least $p-s$ of $A_s$. Thus, the conclusions hold.
	\end{proof}
	
	
	\begin{lemma}[\cite{L}]\label{Clunie}
		Let $f$ be a transcendental meromorphic solution of
		$$
		f^{n}Q^*(z,f)=Q(z,f),
		$$
		where $Q^*(z,f)$ and $Q(z,f)$ are polynomials in $f$ and its derivatives with meromorphic coefficients, say $\left\lbrace a_{\lambda}:\lambda \in I\right\rbrace $, such that $m(r,a_{\lambda})=S(r,f)$ for all $\lambda \in I$. If $\gamma_Q\leq n$, then
		$$
		m(r,Q^*(z,f))=S(r,f).
		$$
	\end{lemma}
	\begin{rem}\label{rem:rationals}
		By using a similar method as in \cite[p. 40]{L}, we can prove that
		$$
		m\left(r, Q^*(z, f)\right)=O(\log r)
		$$
		when $Q^*(z, f)$ and $Q(z, f)$ are differential polynomials in $f$ with rational coefficients.
	\end{rem}

	\begin{proof}[Proof of Theorem~\ref{t1}]
		Let 
		\begin{equation*}
			L(f) := h^{(p)} + r_{p-1}(z) h^{(p-1)} + \cdots + r_1(z) h' + r_0(z) h. 
		\end{equation*}
		From \eqref{nde} and \eqref{odeh}, we have
		$$
		L(f^n) + L(P(z,f)) = L(h)=r_p(z),
		$$
		hence
		\begin{equation}\label{p1}
			(f^n)^{(p)}+r_{p-1}(z)(f^n)^{(p-1)} + \cdots + r_1(z)(f^n)' +r_0(z)f^n = r_p(z)- L(P(z,f)).
		\end{equation}
		Set $Q(z,f)= r_p(z)- L(P(z,f))$, and for an integer $j\ge 1$ define 
		\begin{equation}\label{p2}
			\varphi_j =\left( \frac{(f^n)^{(p)}}{f^n} + r_{p-1}(z) \frac{(f^n)^{(p-1)}}{f^n} + \cdots + r_1(z) \frac{(f^n)'}{f^n} +r_0(z) \right) f^j.
		\end{equation} 
		
		1) If $\varphi_j \equiv 0$, then $L(f^n)=0$ which yields $f$ is of finite order. Since $n\ge p$, i.e., $f^n$ has only zeros of multiplicity at least $n\ge p$, it follows from Lemma~\ref{zerospoles} that $f^n$ has at most finitely many poles and zeros, and so is $f$. Hence, $f$ must have the form $f=qe^\alpha$, where $q$ is non-zero rational and  $\alpha$ is non-constant	polynomial. Substituting this form into \eqref{nde}  we get
		$$
		h(z)= q^n e^{n\alpha(z)}+ \sum_{s=0}^{\gamma_P} \beta_s e^{s \alpha(z)}, 
		$$
		where $\beta_j$ depends on the coefficients of $P(z,f)$, $q$ and $\alpha$ and their derivatives.  By applying Mohon'ko's Lemma \cite[p.~29]{L}, we obtain
		$$
		T(r,h)= nT(r,e^\alpha)  + S(r,f) = nT(r,f)  + S(r,f).
		$$
		
		If $r_s$ ($s=0,\cdots, p-1$) are polynomials, then $f^n$ satisfies $L(f^n)=0$ and it follows from Lemma~\ref{zerospoles} that $f^n$ is entire with only zeros of multiplicity at most $p-1$ and that is not possible since $n\ge p$. Thus, $f^n$ does not have zeros, and then $q$ must be a constant.

		2) If $\varphi_j \not\equiv 0$, then the equation \eqref{p1} will be written as 
		\begin{equation} \label{p3}
			f^{n-j} \varphi_j = Q(z,f),
		\end{equation}
		Using the lemma on the logarithmic derivative, yields obviously that the coefficient of $f^j$ in  \eqref{p2} has small proximity function. If we choose $j$ such that $n\ge j+1$ and $\gamma_Q = \gamma_P \le n-j$, then it follows from Lemma~\ref{Clunie} that $m(r,\varphi_j) = S(r,f)$. Using this fact together with the first main theorem of Nevanlinna, we obtain
		\begin{equation}\label{p4}
			j m(r,1/f) \le m(r,1/\varphi_j)+ S(r,f) = N(r,\infty,\varphi_j) - N(r,0,\varphi_j) + S(r,f).
		\end{equation}
		
		Since any pole of \( f \) is a pole of multiplicity \( p \) of the coefficient of \( f^j \) in \eqref{p2}, it follows that any pole of \( f \) of multiplicity \( k \), which is not a pole or zero of \( r_0, r_1, \dots, r_{p-1} \), is a pole of \( \varphi_j \) of multiplicity \( p + jk \).  
		Moreover, with a more careful inspection and using the fact that \( n \geq p \), we observe that any zero of \( f \), which is not a pole or zero of \( r_0, r_1, \dots, r_{p-1} \), and has multiplicity  \( k \leq \left\lfloor \frac{p-1}{j} \right\rfloor = \left\lceil \frac{p}{j} \right\rceil - 1 \), is a pole of \( \varphi_j \) of multiplicity \( p - jk \).  
		Thus, the poles of \( \varphi_j \) arise only from the poles and zeros of \( f \), and from the poles of the coefficients. Therefore,
		\begin{eqnarray}\label{p5}
			N(r,\infty, \varphi_j) &= &p \overline{N}_{\left.\lceil\frac{p}{j}\rceil -1 \right)}(r,0,f) - j N_{\left.\lceil\frac{p}{j}\rceil -1 \right)}(r,0,f)\nonumber\\
			& +& p \overline{N}(r,\infty,f) + j N(r,\infty,f) + O(\log r).
		\end{eqnarray}
		
		Now, if a zero of multiplicity $k \ge \lfloor \frac{p}{j} \rfloor + 1$ of $f$, which is not a zero or a pole of $r_0,r_1,\dots, r_{p-1}$, is a zero of $\varphi_j$ of multiplicity $jk-p$, hence
		\begin{equation}\label{p6}
			j N_{\left(\lfloor \frac{p}{j} \rfloor + 1 \right.}(r,0,f) -p \overline{N}_{\left(\lfloor \frac{p}{j} \rfloor + 1 \right.}(r,0,f) \le N(r,0,\varphi_j).
		\end{equation}
		
		From \eqref{p4}--\eqref{p6}, we obtain
		\begin{align*}
			j T(r,f) =& j N(r,0,f) + j m(r,f)+O(1) \\
			\le & j N_{\left.\lfloor \frac{p}{j} \rfloor \right)}N(r,0,f) +j N_{\left(\lfloor \frac{p}{j} \rfloor+ 1 \right.}(r,0,f) + N(r,\infty,\varphi_j) - N(r,0,\varphi_j) + S(r,f) \\
			\le & j N_{\left.\lfloor \frac{p}{j} \rfloor \right)}N(r,0,f) + p \overline{N}_{\left(\lfloor \frac{p}{j} \rfloor + 1 \right.}(r,0,f)+N(r,\infty,\varphi_j)+ S(r,f) \\
			=& j N_{\left.\lfloor \frac{p}{j} \rfloor \right)}N(r,0,f) + p \overline{N}_{\left(\lfloor \frac{p}{j} \rfloor + 1 \right.}(r,0,f)+  p \overline{N}_{\left.\lceil\frac{p}{j}\rceil -1 \right)}(r,0,f) - j N_{\left.\lceil\frac{p}{j}\rceil -1 \right)}(r,0,f) +\\
			& p\overline{N}(r,\infty,f) + j N(r,\infty,f) + S(r,f)\\
			=&  p \overline{N} (r,0,f) + p \overline{N}(r,\infty, f) + j N(r,\infty, f) + S(r,f),
		\end{align*}
		therefore, 
		$$
		T(r,f) \le \frac pj \overline{N} (r,0,f) + \frac pj \overline{N}(r,\infty,f) + N(r,\infty, f) + S(r,f).
		$$
		This last inequality is valid and more sharper if we take $$j=\max\{s\in\mathbb{N}:    s\le n-1 \; \text{and} \;  s\le n-\gamma_P\} = \min\{n-1; n-\gamma_P\}=n-\gamma_P$$ since we assume that $\gamma_P \ge 1$. Thus, the conclusion in Theorem~\ref{t1}(2) follows.
	\end{proof}

	\subsection{Proof of Theorem~\ref{Osc}.}
	Dividing both sides of \eqref{L} by \( f \)  and using \cite[Lemma~3.5]{H} yield
	\begin{equation} \label{Q}
		g^n + Q(z, g) = -h,
	\end{equation}  
	where \( g = f'/f \) and 
	$$
	Q(z,g)= p_{n-1}(z)g^{n-1}+\frac{n(n-1)}{2} g^{n-2} g^{\prime} +Q_{n-2}(g),
	$$
	where $Q_{n-2}(g)$ is a polynomial in $g$ and its derivatives with polynomial coefficients and of total degree $\le n-2$.  In particular we have $ \gamma_Q = n-1 $. From \eqref{Q}, we deduce the two cases, corresponding to those in Theorem~\ref{t1}.
	
	\paragraph{\it Case~$ 1 $.}  We have \( g = e^{\alpha} \), where \(\alpha\) is a non-constant polynomial and
	\begin{equation*}\label{e1}
		T(r,h) = nT(r,e^{\alpha}) + S(r,e^{\alpha}),
	\end{equation*}
	from which we conclude  \(\deg(\alpha) = \rho(h)\).
	By letting $ G $ be a primitive of $ e^{\alpha} $, we obtain from \( e^{\alpha} =g = f'/f \), that $ f = e^G $.

	\paragraph{\it Case~$ 2 $.} The function $g$ satisfies the inequality 
	\begin{align}\label{equ:F0}
		m(r, g) &\leq  \frac{p}{n - \gamma_Q} \left( \overline{N}(r,\infty, g) + \overline{N}(r, 0,g) \right) + S(r, g) =  p \overline{N}\left(r, 0,{ff'} \right)+ S(r, g),
	\end{align}
	since \(\gamma_Q = n - 1\) and the zeros of \(ff'\) are precisely the poles of \(g\) (zeros of \(f\)) together with the zeros of \(g\) (points where \(f' = 0\) but \(f \neq 0\)).
	From \eqref{Q} we have  
	\begin{equation}\label{F1}  
		T(r,h) = m(r,h) \leq n m(r,g) + S(r,g),
	\end{equation}  
	which shows, in particular, that $ g $ is a transcendental meromorphic function. Additionally, we have  
	\begin{align*}  
		nm(r,g) = m(r,g^n)  \leq m(r,h) + (n-1)m(r,g) + S(r,g),
	\end{align*}  
	from which we obtain 
	\begin{equation}\label{F2}
		T(r,g) \leq T(r,h) + N(r,\infty,g) + S(r,g) =T(r,h) + \overline{N}\left(r,0,f\right) +S(r,g).  
	\end{equation}
	%
	%
	%
	%
	Combining \eqref{equ:F0}, \eqref{F1} and \eqref{F2} yields \eqref{res}.
	%
	
	If $ \overline{\lambda}(f) < +\infty $, then $ \rho(g) < +\infty $. From Remark~\ref{rem:finitecoeff}, $ S(r,g) $ in  \eqref{equ:F0} is in fact $ O(\log r) $ without an exceptional set. In addition, $ S(r,g) $ in \eqref{F1} is $ O(\log r) $ without an exceptional set. From these two estimates, we obtain \eqref{res1}
	
	\subsection{Proof of Corollary~\ref{cor}}
	
	We already know  from Proposition~\ref{Pro} that \( f \) cannot take the form \( f = e^G \), where \( G' = e^\alpha \) and \( \alpha \) is a polynomial, if \( r_p \not\equiv 0 \).
	Consequently, \( f \) satisfies \eqref{res}. Since \( h \) is of regular growth \cite[p.~108]{Valiron}, it follows \( \overline{\lambda}(f) < \mu(h) \), where $ \mu(h) $ is the lower order of growth of $ h $. Hence,  
	\[
	\overline{N}\left(r,0,{f}\right) = o(T(r, h)), \quad r \to \infty,
	\]  
	from which \eqref{res2} follows immediately.
	
	Now, by Hadamard's factorization theorem, we have \( f = \pi e^g \), where \( \pi \) and \( g \) are entire functions, where \( \lambda(f) = \rho(\pi) < \rho(h) = \rho_2(f) = \rho(g).\)
	Substituting \( f = \pi e^g \) into \eqref{L} results in 
	\begin{eqnarray}\label{Tum}
		(g')^n + Q(z, g') = -h,
	\end{eqnarray}  
	where \( Q(z, g') \) is a polynomial in \( g' \) and its derivatives, with degree \( \gamma_Q \leq n-1 \) and coefficients involving \( \pi^{(k)}/\pi \). Note that  
	\begin{eqnarray}  
		T(r, h) = m(r, h) \leq n m(r, g') + O(\log r)  
		= n T(r, g') + O(\log r).  \label{ineq1}
	\end{eqnarray}  
	On the other hand,
	\begin{eqnarray*}
		T\left( r,\frac{\pi^{(k)}}{\pi}\right) &=&O\left( \overline{N}\left( r,0,\pi\right)\right)  +O(\log r)\\
		&=&O\left( \overline{N}\left( r,0,f\right) \right) +O(\log r)\\
		&=& o\left( T(r,h)\right) +O(\log r),\quad r\to \infty.
	\end{eqnarray*}
	Combining this with \eqref{ineq1}, we obtain \( T(r, \pi^{(k)}/\pi) = S(r, g') \). 
	
	We can now apply Theorem~\ref{t1} to equation \eqref{Tum}. Either \( g' = e^{\alpha} \), where \( \alpha \) is a non-constant polynomial, or \( g' \) satisfies \eqref{res3}. However, the case \( g' = e^{\alpha} \) is not possible; substituting this into \eqref{Tum} and then into \eqref{odeh} leads to a contradiction with \( r_p \not\equiv 0 \). Thus, we conclude that \( g' \) satisfies \eqref{res3}.
	
	\subsection{Proof of Corollary~\ref{Comp}}
	From Theorem~\ref{Comp0}, we know that \( f = e^g \), where \( g \) is an entire function satisfying 
	\begin{eqnarray}\label{rho}
		\rho(g) = \rho_2(f) = \deg(P),
	\end{eqnarray} 
	and \( Q \) must be a nonzero polynomial. Hence, \( h \) satisfies the non-homogeneous linear differential equation  
	\[
	h' - P'(z)h = Q'(z) - Q(z)P(z).
	\]
	Moreover, by applying Corollary~\ref{cor}, we see that \( g' \) satisfies \eqref{res3} for \( n = 2 \), which implies that \( \rho(g') \leq \lambda(g') \).

	Next, we show that \( \lambda(g') \leq \lambda(g) \). Since \( g \) is of finite order, we have 
	\[
	m\left( r,\frac{1}{g}\right) \leq m\left( r,\frac{1}{g'}\right)+O(\log r).
	\]
	Hence,
	\[
	T(r,g) - N(r,0,g) \leq T(r,g') - N(r,0,g') + O(\log r).
	\]
	That is,
	\begin{eqnarray}
		T(r,g) + N(r,0,g') &\leq& T(r,g') + N(r,0,g) + O(\log r) \nonumber\\
		&\leq& T(r,g) + N(r,0,g) + O(\log r). \nonumber
	\end{eqnarray}
	Therefore,
	\[
	N(r,0,g') \leq N(r,0,g) + O(\log r).
	\]
	It follows that
	\[
	\rho(g) = \rho(g') \leq \lambda(g') \leq \lambda(g) \leq \rho(g).
	\]
	Combining this with \eqref{rho} completes the proof of Corollary~\ref{Comp}.


\bigskip
E-mail: \texttt{amine.zemirni@nhsm.edu.dz}

\textsc{National Higher School of Mathematics, Scientific and Technology Hub of Sidi Abdellah, P.O.~Box 68, Algiers 16093, Algeria}

\medskip
E-mail: \texttt{z.latreuch@squ.edu.om}

\textsc{Sultan Qaboos University,
	College of Science,
	Department of Mathematics,
	P.~O.~Box 36, Al-Khod 123,
	Muscat,
	Sultanate of Oman}


\begin{thebibliography}{99}
		
		\bibitem{BLL} 
		S. Bank, I. Laine  and J. Langley,
		\emph{Oscillation results for solutions of linear differential equations in the complex domain},
		Results Math. \textbf{16} (1989), no. 1–2, 3--15.
		
		
		\bibitem{CL} 
		J-F. Chen  and G.~Lian,
		\emph{Expressions of meromorphic solutions of a certain type of nonlinear complex differential equations},
		Bull. Korean Math. Soc. \textbf{57} (2020), no. 4, 1061--1073.
		
		\bibitem{C} 
		J. Clunie,
		\emph{On integral and meromorphic functions},
		J.~London Math.~Soc. \textbf{37} (1962), 17--27.
		
		\bibitem{FC}
		Y-Y. Feng and J-F. Chen,
		\emph{Meromorphic solutions of a certain type of nonlinear differential equations},
		Acta Math. Vietnam. \textbf{49} (2024), no.~2, 173--186.
		
		\bibitem{H} 
		W.~K. Hayman, 
		\emph{Meromorphic Functions}.
		Oxford Mathematical Monographs Clarendon Press, Oxford, 1964.
		
		\bibitem{HKL}
		J. Heittokangas,  R. Korhonen and I. Laine, 
		\emph{On meromorphic solutions of certain nonlinear differential equations},
		Bull. Austral. Math. Soc. \textbf{66} (2002), no.~2, 331--343.
		
		
		\bibitem{HLWZ} 
		J. Heittokangas, Z. Latreuch, J. Wang and M.~A. Zemirni,
		\emph{On meromorphic solutions of non-linear differential equations of Tumura-Clunie type}, 
		Math. Nach. \textbf{294} (2021), no.~4, 748--773.
		
		\bibitem{HYZ} 
		J. Heittokangas,  H. Yu and M.~A. Zemirni,
		\emph{On the number of linearly independent admissible solutions to linear differential and linear difference equations}, 
		Can. J. Math. \textbf{73} (2021), no.~6, 1556--1591.
		
		\bibitem{HILT} 
		J. Heittokangas, K. Ishizaki, I. Laine and K. Tohge,
		\emph{ Exponential polynomials in the oscillation theory}, 
		J. Differ. Equ. \textbf{272} (2021), 911--937.
		
		
		\bibitem{L} 
		I. Laine,
		\emph{Nevanlinna theory and complex differential equations}.
		Studies in Mathematics 15, Walter de Gruyter, Berlin-New York, 1993.
		
		\bibitem{Li2} 
		P. Li,
		\emph{Entire solutions of certain type of differential equations II},
		J.~Math. Anal.~Appl.~\textbf{375} (2011), no.~1, 310--319.
		
		\bibitem{LM} 
		H. Liu and Z. Mao,
		\emph{On the forms of meromorphic solutions of some type of non-linear differential equations},
		Comput. Methods Funct. Theory \textbf{24} (2024), no.~4, 753--768.
		
		\bibitem{LYZ}
		L-W. Liao, C-C.~Yang and J-J.~Zhang,
		\emph{On meromorphic solutions of certain type of non-linear differential equations},
		Ann.~Acad.~Sci.~Fenn. Math.~\textbf{38} (2013), no.~2, 581--593.
		
		\bibitem{Liao}
		L-W. Liao,
		\emph{Non-linear differential equations and Hayman's theorem on differential polynomials},
		Complex Var.~Elliptic Equ.~\textbf{60} (2015), no.~6, 748--756.
		
		\bibitem{LLW}
		X-Q. Lu, L-W.~Liao and J.~Wang,
		\emph{On meromorphic solutions of a certain type of nonlinear differential equations},
		Acta Math.~Sin.~(Engl.~Ser.) \textbf{33} (2017), no.~12, 1597--1608.
		
		\bibitem{MS} 
		E.~Mues and N. Steinmetz,
		\emph{The theorem of Tumura-Clunie for meromorphic functions},
		J.~London Math.~Soc.~(2) \textbf{23} (1981), no.~1, 113--122.
		
		\bibitem{R} 
		I. Reyzl,
		\emph{On the theorem of Tumura-Clunie},
		Complex Variables Theory Appl.~\textbf{28} (1995), no.~2, 175--188.
		
		\bibitem{T} 
		Y. Tumura,
		\emph{On the extensions of Borel's theorem and Saxer-Csillag's theorem},
		Proc.~Phys.~Math.~Soc.~Japan~\textbf{19} (1937), no.~3, 29--35. 
		
		
		\bibitem{Valiron} 
		G. Valiron, 
		\emph{Lectures on the General Theory of Integral Functions}. 
		Chelsea Publishing Company, New York, 1949.
		
		
		\bibitem{Y} 
		C-C. Yang,
		\emph{Applications of the Tumura-Clunie theorem},
		Trans.~Amer. Math.~Soc.~\textbf{151} (1970), 659--662.
		
		\bibitem{Y2} 
		C-C. Yang, 
		\emph{On entire solutions of a certain type of nonlinear differential equation},
		Bull. Austral. Math. Soc. \textbf{64} (2001), no.~3, 377--380.
		
		
		\bibitem{YL}
		C-C. Yang and P.~Li,
		\emph{On the transcendental solutions of a certain type of nonlinear differential equations},
		Arch.~Math.~(Basel)~\textbf{82} (2004), no.~5, 442--448.
		
		\bibitem{YY} 
		C-C. Yang and H-X.~Yi,
		\emph{Uniqueness Theory of Meromorphic Functions}.
		Science Press and Kluwer Acad. Publ., Beijing, 2003.
		
		\bibitem{Yi} 
		H-X.~Yi,
		\emph{On a theorem of Tumura-Clunie for a differential polynomial},
		Bull.~London Math.~Soc.~\textbf{20} (1988), no.~6, 593--596.
		
		\bibitem{ZH} 
		R.~R. Zhang and Z.~B. Huang,
		\emph{On meromorphic solutions of non-linear difference equations}, 
		Comput. Methods Funct. Theory \textbf{18} (2018), no.~3, 389--408.
		
		\bibitem{Z}
		Y.~ Zhang,
		\emph{On entire function $ e^{p(z)}\int_0^z\beta(t) e^{-p(t)} dt $ with applications to Tumura-Clunie equations and complex dynamics},
		Comput. Methods Funct. Theory \textbf{23} (2023), no.~2, 213--235.
	\end{thebibliography}
\end{document}